\documentclass[11pt]{article}

\usepackage{amsmath, amssymb, amsthm}
\usepackage{geometry}
\usepackage{graphicx}
\usepackage{titling}

\title{A Complete, Bounded and Almost Convex Metric Space which is Not Totally Convex}
\author{Carlyle Stewart}

\newtheorem{theorem}{Theorem}[section]

\begin{document}

\maketitle

\begin{center}
\begin{minipage}{0.85\textwidth}
\small
\textbf{Abstract.}
In the 2006 paper ``Exponentiation in V-categories'' by M. M. Clementino
and D. Hofmann, exponentiable objects in the category of metric spaces
with non-expansive maps are characterised as those metric spaces which are
bounded and almost convex. A corollary to this characterisation is made
claiming that every complete, exponentiable metric space is totally convex.
A counterexample to this claim is provided in the present paper, and a correct
version of the claim is suggested and proved.
\end{minipage}
\end{center}

\vspace{1em}

\section{Introduction}
In this paper, we construct a counterexample to a corollary in \cite{ref1}.
In Section 2, this counterexample is constructed, and it is proved that it possesses the desired properties. In Section 3, a modified version of the corollary is stated and proved. Recall that a metric space $(X,d)$ is said to be \textit{almost convex} if, given any $x,y \in X$ and positive real numbers $r,s$ such that
$r+s = d(x,y)$, and any $\varepsilon > 0$, one can find a point $z \in X$
such that
\begin{align*}
d(x,z) &\leq s + \varepsilon, \\
d(y,z) &\leq r + \varepsilon.
\end{align*}
A metric space $(X,d)$ is said to be \textit{totally convex} if, given $x,y \in X$ and positive real numbers $r,s$ such that $r+s = d(x,y)$, there exists a point $z \in X$ such that
\begin{align*}
d(x,z) &= s, \\
d(y,z) &= r.
\end{align*}

In \cite{ref1}, exponentiable objects in the category of metric spaces with non-expansive maps are characterised as those metric spaces which are bounded and almost convex.  Corollary C in Section 4.3 then claims that the following conditions are equivalent for a complete metric space $(X,d)$:
\begin{enumerate}
\item $X$ is exponentiable.
\item $X$ is bounded and totally convex.
\end{enumerate}
In particular, this implies that every bounded, complete, almost convex metric space is totally convex. We provide a counterexample to this claim (see Theorem 2.1) and prove that the claim holds under the additional assumption of compactness (see Theorem 3.1).

\vfill

\noindent\rule{\textwidth}{0.4pt}

\small
\noindent Department of Mathematical Sciences, Stellenbosch University
\section{A counterexample}
We define a subspace of $\mathbb{R}^2$, and then remetrise it with the induced path-length metric. To define the subspace, first define the ``hat'' of height $\varepsilon > 0$,
$H(\varepsilon)$, to be the union of the line segments (with endpoints)
connecting $(0,0)$ to $(1,\varepsilon)$ and $(1,\varepsilon)$ to $(2,0)$. The hat space is then
\[
X = \bigcup_{n \in \mathbb{N}} H\left(\frac{1}{n}\right),
\]
with the induced path-length metric.
\begin{figure}[h]
\centering
\includegraphics[width=0.5\textwidth]{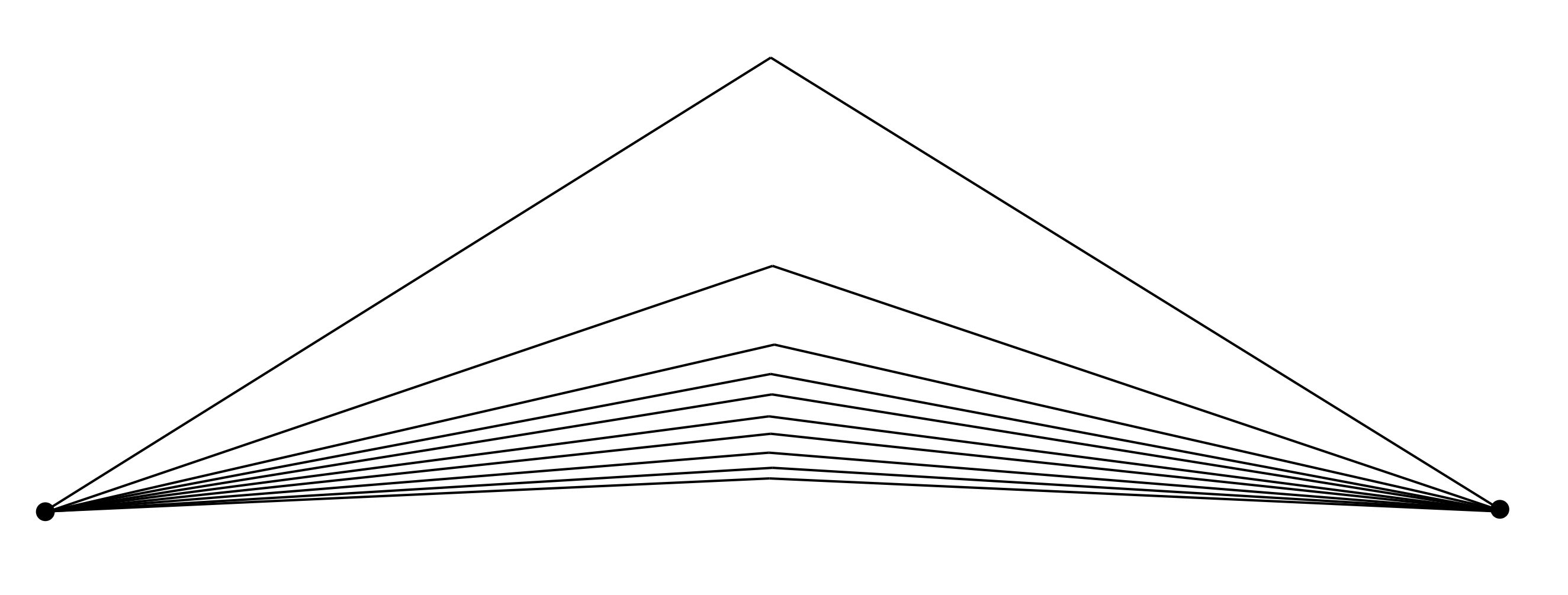}
\caption{An illustration of the hat space}
\end{figure}
\begin{theorem}
The hat space is bounded, complete and almost convex, but not totally convex.
\end{theorem}
\begin{proof}
Note that the distance between the points $(0, 0)$ and $(2, 0)$ in the hat space is still 2, since the hat $H(\varepsilon)$ contains a polygonal path connecting these two points, of length $2 + 2\varepsilon^2$. The hat space is not totally convex, since there is no point $x \in X$
satisfying
\begin{align*}
d((0,0),x) &= 1, \\
d((2,0),x) &= 1.
\end{align*}
The hat space is however complete, bounded and almost convex. It is easy to see that the hat space is bounded and almost convex. To show that it is complete, let $(x_n)_{n \in \mathbb{N}}$ be a Cauchy sequence in $(X,d)$. Recall that
\[
X = \bigcup_{n \in \mathbb{N}} H\left(\frac{1}{n}\right),
\]
and then note that if $(x_n)$ is eventually contained in some $H(1/m)$, then it converges since each hat is complete (as a subspace). Now suppose that $(x_n)$ does not converge to $(0,0)$ or $(2,0)$. Then there is some $\varepsilon > 0$ and a subsequence $(x_{n_k})$ such that
\[
d\big((0,0), x_{n_k}\big) > \varepsilon \quad \text{and} \quad d\big((2,0), x_{n_k}\big) > \varepsilon
\]
for all $k \in \mathbb{N}$. Now, if $(x_{n_k})$ is eventually confined to a single hat, then it is convergent, and hence $(x_n)$ is convergent. Thus, suppose $(x_{n_k})$ is not eventually confined to a single hat. Then for any $m \in \mathbb{N}$, we can find $l > m$ such that $x_{n_m}$ and $x_{n_l}$ lie on different hats. But then
\[
d(x_{n_m}, x_{n_l}) > 2\varepsilon,
\]
since the metric on $X$ is the path-length metric. This contradicts the assumption that $(x_n)$ is Cauchy. Hence $(x_n)$ converges, and therefore $X$ is complete.
\end{proof}

\section{A correction}

\begin{theorem}
A compact, almost convex metric space $(X,d)$ is totally convex.
\end{theorem}

\begin{proof}
Let $x,y \in X$ and $r,s \geq 0$ such that $r+s = d(x,y)$. Since $X$ is almost convex, we can construct a sequence $(z_n)$ in $X$ such that
\begin{align*}
d(x,z_n) &\leq r + \frac{1}{n}, \\
d(y,z_n) &\leq s + \frac{1}{n}.
\end{align*}
By compactness, $(z_n)$ has a convergent subsequence $(z_{n_m})$.
Let
\[
z = \lim_{m \to \infty} z_{n_m}.
\]
By continuity of the metric,
\begin{align*}
d(x,z) &= \lim_{m \to \infty} d(x,z_{n_m}) \leq r, \\
d(y,z) &= \lim_{m \to \infty} d(y,z_{n_m}) \leq s.
\end{align*}
Hence,
\[
d(x,z) = r, \quad d(y,z) = s,
\]
and $X$ is totally convex.
\end{proof}

\section*{Acknowledgments}

The author extends their gratitude to Graham Manuel for suggesting the
corrected version of the relevant claim, and to Zurab Janelidze for assisting
in the editing of the present text.

\end{document}